\documentclass{amsart}
\usepackage{graphicx}
\usepackage{color}
\usepackage{epstopdf}

\newtheorem{thm}{Theorem}
\newtheorem{cor}[thm]{Corollary}
\newtheorem{lem}[thm]{Lemma}
\newtheorem{prop}[thm]{Proposition}
\theoremstyle{remark}
\newtheorem{rem}[thm]{Remark}

\begin{document}
\title[Saddle point solutions for non-local elliptic operators]
{Saddle point solutions \\for non-local elliptic operators}
\author{Alessio Fiscella}
\address[A.~Fiscella]{Dipartimento di Matematica ``Federigo Enriques'', Universit\`a di Milano\\
Via Cesare Saldini, 50  20133 Milano ITALY}
\email{alessio.fiscella@unimi.it}

\date{\today}

\begin{abstract}
The paper deals with equations driven by a non-local integrodifferential operator $\mathcal L_K$ with homogeneous Dirichlet boundary conditions.
These equations have a
variational structure and we find a solution for them
using the Saddle Point Theorem. We prove this result for a general
integrodifferential operator of fractional type
and from this, as a particular case, one can derive an
existence theorem for the fractional Laplacian, finding 
solutions of the equation
$$ \left\{
\begin{array}{ll}
(-\Delta)^s u=f(x,u) & {\mbox{ in }} \Omega\\
u=0 & {\mbox{ in }} \mathbb{R}^n\setminus \Omega\,,
\end{array} \right.
$$
where the nonlinear term $f$ satisfies a linear growth condition.
\end{abstract}
\maketitle
\section{Introduction}
In this paper we deal with the following problem
\begin{equation}
\mbox{
$\left\{\begin{array}{ll}
$$-\mathcal L_Ku=f(x,u) & \mbox{in } \Omega,$$\\

$$u=0 & \mbox{in } \mathbb{R}^{n}\setminus\Omega$$
\end{array}
\right.$}\label{P}
\end{equation}
where $\Omega\subset\mathbb{R}^{n}$ is an open and bounded set, $f:\Omega\times\mathbb{R}\rightarrow\mathbb{R}$ is a Carath\'eodory function whose properties will be introduced later and $\mathcal L_K$ is a non-local operator formally defined as follows:
\begin{equation}
\mathcal L_Ku(x)=\frac{1}{2}\int_{\mathbb{R}^{n}}(u(x+y)+u(x-y)-2u(x))K(y)dy,
\end{equation}
for any $x\in\mathbb{R}^{n}$, where $K:\mathbb{R}^{n}\setminus \left\{0\right\}\rightarrow(0,+\infty)$ is a given function.
In the case where $K(x)=\left|x\right|^{-(n+2s)}$, for a given $s\in(0,1)$, there are several studies about this problem (see \cite{valpal} and references therein). In this case problem \eqref{P} becomes
\begin{equation}
\mbox{
$\left\{\begin{array}{ll}
$$(-\Delta)^{s} u=f(x,u) & \mbox{in } \Omega,$$\\

$$u=0 & \mbox{in } \mathbb{R}^{n}\setminus \Omega,$$
\end{array}
\right.$}\label{p}
\end{equation}
where $-(-\Delta)^{s}$ is the fractional Laplace operator
which (up to normalization factors) may be formally defined as
\begin{equation}\label{laplacian}
-(-\Delta)^s u(x)=
\frac{1}{2}
\int_{\mathbb{R}^{n}}\frac{u(x+y)+u(x-y)-2u(x)}{|y|^{n+2s}}\,dy,
\end{equation}
for any $x\in\mathbb{R}^{n}$.

As we said, here problems \eqref{P} and \eqref{p} are only expressed in a formal way. In classical terms, the definition in \eqref{laplacian} makes sense if $u\in C^2_0(\Omega)$, for example. However, under suitable assumptions on $f$ and $K$, we can express problems \eqref{P} and \eqref{p} in a variational form which allows us to give a simple and complete explanation and also to set the study of \eqref{P}. In this way problem \eqref{P} becomes the Euler-Lagrange equation of a suitable functional defined in a suitable space.

For this, we assume that $K$ satisfies the following conditions:
\begin{equation}\label{K1}
mK\in L^{1}(\mathbb{R}^{n}), \,\,\mbox{where}\,\, m(x)=\min\left\{\left|x\right|^{2},1\right\};
\end{equation}
\begin{equation}\label{K2}
\begin{alignedat}2
&\mbox{there exists}\,\, \theta >0 \,\,\mbox{and}\,\, s\in (0,1) \,\,\mbox{such that}\,\, K(x)\geq\theta\left|x\right|^{-(n+2s)}\\
&\mbox{for any}\,\, x\in\mathbb{R}^{n}\setminus\left\{0\right\}.
\end{alignedat}
\end{equation}

The assumptions of the function $f$ have a direct influence on the topological structure of the problem.
When the function $f$ satisfies superlinear and subcritical growth conditions, the functional associated to problem \eqref{P} satisfies the geometry of the Mountain Pass Theorem; see for example \cite{sv2}.
In \cite{sv5} the right-hand side of equation \eqref{P} is equal to $f(x,u)+\lambda u$, where $\lambda$ is a real parameter and the nonlinear term $f$ satisfies superlinear and subcritical
growth conditions. In this case critical points of the Euler-Lagrange functional can be obtained by using both the Mountain Pass Theorem and the Linking Theorem, depending on the value of $\lambda$.

In view of our problem we assume that, in addition to the usual Carath\'eodory conditions, $f$ also satisfies the following condition:
\begin{equation}\label{f1}
\begin{alignedat}2
&\mbox{there exist} \,\,a\in L^{2}(\Omega)\,\, \mbox{and}\,\, b\geq 0\,\, \mbox{such that}\,\, \left|f(x,t)\right|\leq a(x)+b\left|t\right|\\
&\mbox{for any}\,\,t\in\mathbb{R}\,\,\mbox{and a.e.}\,\, x\in \Omega.
\end{alignedat}
\end{equation}

Now, we introduce the functional spaces. Here, the functional space $X$ denotes the linear space of Lebesgue measurable functions $u:\mathbb{R}^n\rightarrow\mathbb{R}$ such that 
	\[\mbox{the map}\,\,\,
(x,y)\mapsto (u(x)-u(y))^2 K(x-y)\,\,\, \mbox{is in}\,\,\, L^1\big(Q,\,dxdy\big),
\]
where $Q:=\mathbb{R}^{2n}\setminus({\mathcal C}\Omega\times{\mathcal C}\Omega)$.
The space $X$ is endowed with the norm defined as
\begin{equation}\label{norma}
\left\|u\right\|_{X}=\Big(\int_\Omega \left|u(x)\right|^2dx+\int_Q |u(x)-u(y)|^2K(x-y)dx\,dy\Big)^{1/2}\,.
\end{equation}
It is immediate to observe that bounded and Lipschitz functions belong to $X$ (see \cite{svnew, sv2} for further details on space $X$).
Moreover, we denote with $Z$ the closure of $C^{\infty}_{0}(\Omega)$ in $X$.

Now, we can state in a precise way problem \eqref{P} by writing it in the variational form:
\begin{equation}
\mbox{
$\left\{\begin{array}{lll}
$$\displaystyle\int_{Q} (u(x)-u(y))(\varphi(x)-\varphi(y))K(x-y) dx\,dy\\
\qquad \qquad \qquad\qquad\qquad=\displaystyle\int_{\Omega} f(x,u(x))\varphi(x)dx\quad\mbox{for any}\,\,\varphi \in Z$$\\
$$u\in Z$$.
\end{array}
\right.$}\label{wf}
\end{equation}
Thanks to our assumptions on $\Omega$, $f$ and $K$, all the integrals in \eqref{wf} are well defined if $u$, $\varphi\in Z$. We also point out that the odd part of function $K$ gives no contribution to the integral of the left-hand side of \eqref{wf}. Therefore, it would be not restrictive to assume that $K$ is even.

Now, we introduce the main result of the paper. Here, we denote with $\lambda_1, \lambda_2, \ldots$ the eigenvalues of $-\mathcal L_K$ which are briefly recalled in Proposition \ref{autovalori} (see also \cite[Section 3]{sv5}).
\begin{thm} \label{Thsaddle}
Let $\Omega$ be a bounded subset of $\mathbb{R}^{n}$ and let $K$ and $f$ be two functions satisfying assumptions \eqref{K1}--\eqref{f1}. Moreover, by setting
\begin{equation}\label{finf}
\liminf_{\left|t\right|\rightarrow \infty}\frac{f(x,t)}{t}:=\underline{\alpha}(x)\quad \mbox{and} \quad\limsup_{\left|t\right|\rightarrow \infty}\frac{f(x,t)}{t}:=\overline{\alpha}(x)\quad \mbox{for a.e.}\,\, x\in\Omega,
\end{equation}
we assume that one of the two following conditions is satisfied: either $\overline{\alpha}(x)<\lambda_{1}$ for a.e. $x\in\Omega$, or there exists $k\in\mathbb{N}^*$ such that $\lambda_{k}<\underline{\alpha}(x)\leq\overline{\alpha}(x)<\lambda_{k+1}$ for a.e. $x\in\Omega$. Then, problem \eqref{wf} admits a solution $u\in Z$.
\end{thm}
\begin{rem}
We notice that, in our framework,
no solution of problem \eqref{wf}
is known from the beginning,
unlike the cases treated in \cite{sv2, sv5}
where the problems considered admit
the trivial solution $u=0$ (indeed, in our case,
$f(x,0)+h(x)$ may not vanish and $u=0$ may not be a solution).
\end{rem}

The proof of Theorem \ref{Thsaddle} relies on the Saddle Point Theorem (see, for instance, \cite{rabinowitz}). In order to check the geometric assumptions needed for applying this result, we perform some energy estimates in fractional Sobolev spaces.
Indeed, Theorem~\ref{Thsaddle} is the fractional analog of a result valid for
the classical Laplacian (see, e.g.,~\cite[Theorem 4.1.1]{Mugnai}). As a matter of fact,
we plan to consider further applications of the Saddle Point
Theorem for fractional operators for asymptotically linear
terms in a forthcoming paper.

It is an interesting question if weak solutions of problem \eqref{wf} solve also problem \eqref{P} in an appropriate strong sense. 
Some interesting reults about this problem can be found in \cite{sv6} (see also \cite[Theorem 5]{BFV}) and a more exhaustive answer will be provided in a forthcoming paper.

Moreover, it is worth pointing out that the solution found in Theorem \ref{Thsaddle} is unique, under a suitable condition on the nonlinearity.
\begin{cor}\label{cor}
Under the same assumptions of Theorem \ref{Thsaddle} and if in addition there exists a $k\in\mathbb{N}^*$ such that
\begin{equation}\label{f2}
\lambda_k<\frac{f(x,s)-f(x,t)}{s-t}<\lambda_{k+1}\quad \mbox{for any}\,\, s,\, t\in\mathbb{R}\,\, \mbox{with}\,\, s\neq t\,\,\mbox{and a.e.}\,\, x\in\Omega\,,
\end{equation}
then the solution of problem \eqref{wf} is unique.
\end{cor}

The paper is organized as follows. In Section~\ref{sec setting} we introduce the functional setting we will work in and we recall some basic facts on the spectral theory of the operator $\mathcal L_K$. In Section~\ref{sec saddle} we prove
Theorem~\ref{Thsaddle} performing the classical Saddle Point
Theorem.

\section{The functional analytic setting and an eigenvalue problem}\label{sec setting}
At first, we recall some preliminary results on 
the functional space $Z$, introduced on page 2.
\begin{lem}\label{z}
Let $K:\mathbb{R}^n\setminus\{0\}\rightarrow (0,+\infty)$
satisfy assumptions \eqref{K1} and \eqref{K2}.
Then, the following assertions hold true:
\begin{itemize}
\item[$a)$] $Z$ is continuously embedded in $W^{s,2}_{0}(\Omega)$ (for a detailed description see \cite{valpal}) which is the closure of $C^{\infty}_{0}(\Omega)$ in the space $W^{s,2}(\Omega)$ of functions $u$ defined on $\Omega$ for which is well defined the so-called \emph{Gagliardo norm}
\begin{equation*}
\|u\|_{W^{s,2}(\Omega)}=
\Big(\int_\Omega \left|u(x)\right|^2dx+\int_{\Omega\times
\Omega}\frac{|u(x)-u(y)|^2}{|x-y|^{n+2s}}\,dx\,dy\Big)^{1/2}\,.
\end{equation*}
\item[$b)$] $Z$ is compactly embedded in $L^p(\Omega)$ for any $p\in[1,2^*)$, where the fractional critical Sobolev exponent is defined as
\begin{equation*}
2^*:=
\left\{\begin{array}{ll}
$$\displaystyle\frac{2n}{n-2s} & \mbox{if } n>2s,$$\\

$$+\infty & \mbox{if } n\leq 2s.$$
\end{array}\right.
\end{equation*}
\item[$c)$] $Z$ is a Hilbert space endowed with the following norm
\begin{equation}\label{normaZ}
\left\|v\right\|_{Z}=\Big(\int_Q |v(x)-v(y)|^2K(x-y)dx\,dy\Big)^{1/2}\,,
\end{equation}
which is equivalent to the usual one defined in \eqref{norma}.
\end{itemize}
\end{lem}
\begin{proof}
For part $a)$ we simply observe that by \eqref{K2} we get
\begin{equation}\label{theta}
\theta\int_Q\frac{\left|u(x)-u(y)\right|^2}{\left|x-y\right|^{n+2s}}\,dx\,dy\leq\int_Q\left|u(x)-u(y)\right|^2 K(x-y)\,dx\,dy
\end{equation}
and so
$$
\|u\|_{W^{s,2}(\Omega)}\leq c(\theta)\|u\|_X\,,
$$
with $c(\theta)=\max\{1, \theta^{-1/2}\}$.

Now, we prove part $b)$. Let $\Omega'$ be a regular, open subset of $\mathbb{R}^n$ such that $\Omega\subseteq\Omega'$. For any $u\in W^{s,2}_{0}(\Omega)$ we can define
	\[\widetilde{u}(x):=\left\{\begin{array}{ll}
$$u(x) & \mbox{if } x\in\Omega,$$\\

$$0 & \mbox{if } x\in\Omega'\setminus\Omega.$$
\end{array}\right.
\]
It is clear that $\widetilde{u}\in W^{s,2}_{0}(\Omega')$. Indeed, if $\left\{u_j\right\}_{j\in\mathbb{N}}$ is a sequence in $C^\infty_0(\Omega)$ which converges to $u$ in $W^{s,2}_{0}(\Omega)$ then $\left\{\widetilde{u}_j\right\}_{j\in\mathbb{N}}$ is a sequence in $C^\infty_0(\Omega')$ which converges to $\widetilde{u}$ in $W^{s,2}_{0}(\Omega')$.
Moreover, we also have
	\[\left\|\widetilde{u}\right\|_{W^{s,2}(\Omega')}=\left\|u\right\|_{W^{s,2}(\Omega)}.
\]
Thus, $W^{s,2}_{0}(\Omega')$ is isometric embedded in $W^{s,2}_{0}(\Omega)$.
The conclusion follows by remembering that $W^{s,2}_{0}(\Omega')$ is compactly embedded in $L^p(\Omega')$ with $1\leq p <2^*$.

For the assertion $c)$ we claim that there exists a constant $C>0$ such that
\begin{equation}\label{poincarè}
\left\|u\right\|_{L^2(\Omega)}\leq C\left(\int_Q \frac{\left|u(x)-u(y)\right|^2}{\left|x-y\right|^{n+2s}}\,dx\,dy\right)^{1/2}
\end{equation}
for any $u\in W^{s,2}_{0}(\Omega)$.
In fact, since $\Omega$ is bounded there is $R>0$ such that $\Omega\subseteq B_R$ and $\left|B_R\setminus\Omega\right|>0$. So, we get
\begin{equation*}
\begin{alignedat}2
&\int_Q\frac{\left|u(x)-u(y)\right|^2}{\left|x-y\right|^{n+2s}}\,dx\,dy\geq\int_{{\mathcal C}\Omega}\left(\int_\Omega\frac{\left|u(x)-u(y)\right|^2}{\left|x-y\right|^{n+2s}}\,dy\right)dx\\
&=\int_{{\mathcal C}\Omega}\left(\int_\Omega\frac{\left|u(y)\right|^2}{\left|x-y\right|^{n+2s}}dy\right)dx
\geq\int_{B_R\setminus\Omega}\left(\int_\Omega\frac{\left|u(y)\right|^2}{\left|2R\right|^{n+2s}}dy\right)dx=\frac{\left|B_R\setminus\Omega\right|}{(2R)^{n+2s}}\left\|u\right\|^{2}_{L^2(\Omega)}
\end{alignedat}
\end{equation*}
for any $u\in W^{s,2}_{0}(\Omega)$ (since $u=0$ in $\mathbb{R}^n\setminus\Omega$), which proves our claim. Finally, by combining \eqref{theta} and \eqref{poincarè} we conclude the proof.
\end{proof}
From now on, we take \eqref{normaZ} as norm on $Z$.
Now, we study some properties of eigenvalues and eigenfunctions of the non-local operator $-\mathcal L_K$ (for a more general and detailed study see \cite{sv5}).
\begin{prop}\label{autovalori}
Let $K:\mathbb{R}^n\setminus\{0\}\rightarrow (0,+\infty)$
satisfy assumptions \eqref{K1} and \eqref{K2}.
Then, there exists an orthogonal complete basis of eigenvectors $e_j$ $(j=1,2,\ldots)$ in $Z$ normalized in $L^2(\Omega)$, by the quadratic form $\left\|\cdot\right\|^{2}_{L^2(\Omega)}$.
The corrisponding eigenvalues $\lambda^{-1}_{j}$ verify \,$0<\lambda_1<\lambda_2\leq\lambda_3\leq\ldots$
and \,$\displaystyle\sup_{j\in\mathbb{N}^*}\lambda_j=+\infty$.
Moreover, for any $k\in\mathbb{N}^*$ it follows that
\begin{equation}\label{Hk}
\int_Q(u(x)-u(y))^2 K(x-y)\,dxdy\leq\lambda_k\left\|u\right\|^{2}_{L^2(\Omega)}\quad\mbox{for any}\,\,u\in\mbox{span}(e_1,\ldots e_k),
\end{equation}
\begin{equation}\label{Pk}
\int_Q(u(x)-u(y))^2 K(x-y)\,dxdy\geq\lambda_k\left\|u\right\|^{2}_{L^2(\Omega)}\quad\mbox{for any}\,\,u\in\mathbb{P}_k,
\end{equation}
where $\mathbb{P}_{k}:=\left\{u\in Z:\,\left\langle u,e_j\right\rangle_{Z}=0\quad\mbox{for any}\,\,\, j=1,\ldots,k-1\right\}$ $(\mathbb{P}_{1}:=Z)$.
\end{prop}
\begin{proof}
The proof follows by the general theory of functional analysis and by the compact embedding of $Z$ in $L^2(\Omega)$, proved in Lemma \ref{z}. Moreover, the fact that the eigenvalue $\lambda_1$ is simple is proved in \cite[Proposition 9]{sv5}.
\end{proof}

\section{Proof of Theorem~\ref{Thsaddle}}\label{sec saddle}

For the proof of Theorem~\ref{Thsaddle}, we observe that
problem~\eqref{wf} has a variational structure, indeed it is the Euler-Lagrange equation of the functional $\mathcal J:Z\to \mathbb{R}$ defined as follows
$$\mathcal J(u)=\frac 1 2 \int_Q|u(x)-u(y)|^2 K(x-y)\,dx\,dy-\int_\Omega F(x, u(x))dx\,,$$
where $F$ is the primitive of $f$ with respect to the second variable, that is
	\[F(x,t)=\int^{t}_{0}f(x,\tau)d\tau.
\]
\noindent Moreover, note that the functional $\mathcal J$ is Fr\'echet differentiable in $u\in Z$ and for any $\varphi\in Z$
$$\begin{aligned}
\mathcal J'(u)(\varphi) & = \int_Q \big(u(x)-u(y)\big)\big(\varphi(x)-\varphi(y)\big)K(x-y)\,dx\,dy\\
& \qquad \qquad \qquad \qquad \qquad \qquad -\int_\Omega f(x, u(x))\varphi(x)\,dx\,.
\end{aligned}$$
Thus, critical points of $\mathcal J$ are solutions to problem~\eqref{wf}.
In order to find these critical points, we will divide the proof in two cases. At first, when $\overline{\alpha}(x)<\lambda_{1}$ the existence of the solution of problem \eqref{wf} follows from the Weierstrass Theorem (i.e. by direct minimization). When  $\lambda_{k}<\underline{\alpha}(x)\leq\overline{\alpha}(x)<\lambda_{k+1}$ for some $k\in\mathbb{N}^*$, we will make use of the Saddle Point Theorem (see \cite{rabinowitz}).
For this, we have to check that the functional $\mathcal J$ has a particular geometric structure
(as stated, e.g., in conditions~($I_{3}$) and ($I_{4}$) of \cite[Theorem~4.6]{rabinowitz})
and that it satisfies the Palais--Smale compactness condition (see, for instance, \cite[page~3]{rabinowitz}).

\subsection{The case $\overline{\alpha}(x)<\lambda_{1}$}
In this subsection, in order to apply the Weierstrass Theorem we first verify that the functional $\mathcal J$ satisfy the following geometric feature.
\begin{prop}\label{prop coercività}
Let $K$ and $f$ be two functions satisfying assumptions \eqref{K1}--\eqref{f1}. Moreover, let $\overline{\alpha}(x)<\lambda_{1}$ a.e. in $\Omega$. Then, the functional $\mathcal J$ verifies
\begin{equation}\label{coercività}
\liminf_{\left\|u\right\|_{Z}\rightarrow +\infty}\frac{\mathcal J(u)}{\left\|u\right\|^{2}_{Z}}>0.
\end{equation}
\end{prop}
\begin{proof}
It is enough to show that if $\left\{u_{j}\right\}_{j\in\mathbb{N}}$ is a sequence in $Z$ such that $\left\|u_{j}\right\|_{Z}\rightarrow +\infty$, then
\begin{equation}\label{coesempl}	\limsup_{j\rightarrow+\infty}\int_\Omega\frac{F(x,u_j(x))}{\left\|u_j\right\|^{2}_{Z}}dx<\frac{1}{2}.
\end{equation}

By Lemma \ref{z}, up to a subsequence, there exists $u_0\in Z$ such that $u_{j}/\left\|u_{j}\right\|_{Z}$ converges to $u_0$ strongly in $L^2(\Omega)$ and a.e. in $\Omega$, as well as weakly in $Z$. So, $\left\|u_0\right\|_{Z}\leq 1$.
Now, by \eqref{f1} we observe that
\begin{equation*}
\frac{\left|F(x,u_{j})\right|}{\left\|u_{j}\right\|^{2}_{Z}}\leq \frac{a(x)\left|u_{j}\right|+b\frac{\displaystyle\left|u_{j}\right|^{2}}{\displaystyle2}}{\left\|u_{j}\right\|^{2}_{Z}}\,,
\end{equation*}
where the sequence on the right-hand side converges in $L^1(\Omega)$.
Moreover, we claim that
\begin{equation}\label{claimnew}
\limsup_{j\rightarrow +\infty}\frac{F(x,u_{j}(x))}{\left\|u_{j}\right\|^{2}_{Z}}\leq\frac{\overline{\alpha}(x)}{2}\left|u_0(x)\right|^2
\end{equation}
which follows by previous formula when $x\in\Omega$ such that $u_0(x)=0$. While, for $x$ such that $u_0(x)\neq 0$, \eqref{claimnew} follows from the fact that in this case $\left|u_j(x)\right|^2\rightarrow+\infty$ and so for $j$ sufficiently large we get
\[\frac{F(x,u_j(x))}{\left\|u_j\right\|^{2}_{Z}}=\frac{F(x,u_j(x))}{\left|u_{j}(x)\right|^2}\frac{\left|u_{j}(x)\right|^2}{\left\|u_j\right\|^{2}_{Z}},
\]
and also by \eqref{f1} and \eqref{finf} we have
\begin{equation*}
\limsup_{\left|t\right|\rightarrow \infty}\frac{F(x,t)}{t^{2}}\leq\frac{\overline{\alpha}(x)}{2}.
\end{equation*}

Thus, by the generalized Fatou Lemma, \eqref{Pk} and \eqref{claimnew} it follows that
\begin{equation*}
\limsup_{j\rightarrow +\infty}\int_{\Omega}\frac{F(x,u_{j}(x))}{\left\|u_{j}\right\|^{2}_{Z}}dx\leq\int_{\Omega}\frac{\overline{\alpha}(x)}{2}\left|u_0(x)\right|^2dx
\leq\frac{\lambda_1}{2}\int_\Omega\left|u_0(x)\right|^2dx\leq\frac{\left\|u_0\right\|^2_Z}{2}\leq\frac{1}{2}.
\end{equation*}
The second of these last inequalities is strict if $u_0\neq 0$, while the last one is strict if $u_0= 0$.
\end{proof}

\begin{proof}[\bf Proof of Theorem \ref{Thsaddle}, when $\overline{\alpha}(x)<\lambda_{1}$]
As is well known, the map $u\mapsto\left\|u\right\|^{2}_{Z}$ is lower semicontinuous in the weak topology of $Z$, while the map $u\mapsto\int_{\Omega}F(x,u)$ is continuous in the weak topology of $Z$, since \eqref{f1} implies that 
	\[\left|F(x,t)\right|\leq a(x)\left|t\right|+b\frac{\left|t\right|^2}{2}.
\]
So, the functional $\mathcal J$ is lower semicontinuous and by using also \eqref{coercività} to obtain coerciveness we can apply the Weierstrass Theorem in order to find a minimum of $\mathcal J$ on $Z$, which is clearly a solution of problem \eqref{wf}.
\end{proof}

\subsection{The case $\lambda_{k}<\underline{\alpha}(x)\leq\overline{\alpha}(x)<\lambda_{k+1}$}

At first, we recall that, in what follows, $e_{k}$ will be the $k$-th eigenfunction corresponding to the eigenvalue $\lambda_{k}$ of $-\mathcal L_{K}$ for any $k\in\mathbb{N}^*$, and we set
	\[\mathbb{P}_{k+1}:=\left\{u\in Z:\,\,\left\langle u,e_{j}\right\rangle_{Z}=0\quad\mbox{for any}\,\, j=1,\ldots,k \right\}
\]
as defined in Proposition \ref{autovalori}, while $H_{k}:=\mbox{span}\left\{e_{1},\ldots,e_{k}\right\}$  for any $k\in\mathbb{N}^*$. It is immediate to observe that $\mathbb{P}_{k+1}=H^{\bot}_{k}$ with respect to the scalar product in $Z$ and $Z=H_k\oplus\mathbb{P}_{k+1}$. 

Now, we prove that the functional $\mathcal J$ has the geometric features required by the Saddle Point Theorem.
\begin{prop}\label{prop Hk}
Let $K$ and $f$ be two functions satisfying assumptions \eqref{K1}--\eqref{f1}. Moreover, assume there exists $k\in\mathbb{N}^*$ such that  $\lambda_{k}<\underline{\alpha}(x)\leq\overline{\alpha}(x)<\lambda_{k+1}$ a.e. in $\Omega$. Then, the functional $\mathcal J$ verifies
\begin{equation}
\limsup_{u\in H_{k},\, \left\|u\right\|_{Z}\rightarrow +\infty}\frac{\mathcal J(u)}{\left\|u\right\|^{2}_{Z}}<0. \label{limsup}
\end{equation}
\end{prop}

\begin{proof}
Let $\left\{u_{j}\right\}_{j\in\mathbb{N}}$ be a sequence in $H_{k}$ such that  $\left\|u_{j}\right\|_{Z}\rightarrow +\infty$. Since $H_k$ is finite dimensional, there exists $u_0\in H_k$ such that $u_{j}/\left\|u_{j}\right\|_{Z}$ converges to $u_0$ strongly in $Z$ and also $\left\|u_0\right\|_{Z}=1$.

Now, by proceeding as in the proof of claim \eqref{claimnew}, it follows that
\begin{equation*}
\liminf_{j\rightarrow +\infty}\frac{F(x,u_{j}(x))}{\left\|u_{j}\right\|^{2}_{Z}}\geq\frac{\underline{\alpha}(x)}{2}\left|u_0(x)\right|^2,
\end{equation*}
a.e. in $\Omega$.
So, by using also the Fatou Lemma and the fact that $\underline{\alpha}(x)>\lambda_{k}$, we have
\begin{equation*}
\limsup_{j\rightarrow +\infty}\frac{\mathcal J(u_{j})}{\left\|u_{j}\right\|^{2}_{Z}}\leq \frac{1}{2}-\int_\Omega\frac{\underline{\alpha}(x)}{2}\left|u_0(x)\right|^2dx
<\frac{1}{2}-\frac{\lambda_{k}}{2}\int_\Omega \left|u_0(x)\right|^2dx.
\end{equation*}
By the last inequality, \eqref{Hk} and the fact that $\left\|u_0\right\|_{Z}=1$, we get \eqref{limsup}.
\end{proof}
Also, Proposition \ref{prop Hk} has the following counterpart.

\begin{prop}\label{prop Hk ortogonale}
Let $K$ and $f$ be two functions satisfying assumptions \eqref{K1}--\eqref{f1}. Moreover, assume there exists $k\in\mathbb{N}^*$ such that  $\lambda_{k}<\underline{\alpha}(x)\leq\overline{\alpha}(x)<\lambda_{k+1}$ a.e. in $\Omega$. Then, the functional $\mathcal J$ verifies
\begin{equation}
\liminf_{u\in\mathbb{P}_{k+1},\,\left\|u\right\|_{Z}\rightarrow +\infty}\frac{\mathcal J(u)}{\left\|u\right\|^{2}_{Z}}>0. \label{coercività Hk ortogonale}
\end{equation}
\end{prop}
\begin{proof}
The proof is similar to the proof of Proposition \ref{prop coercività}. In this case we have $\overline{\alpha}(x)<\lambda_{k+1}$, for some $k\in\mathbb{N}^*$, instead of $\overline{\alpha}(x)<\lambda_1$.
\end{proof}

In order to prove the boundedness of a Palais-Smale sequence, we first introduce the following lemma.
\begin{lem}\label{lemmanuovo}
Let $K$ be a function satisfying \eqref{K1} and \eqref{K2}. Moreover, assume there exist $k\in\mathbb{N}^*$ and a measurable function $m$ on $\Omega$ such that $\lambda_k<m(x)<\lambda_{k+1}$ for a.e. $x\in\Omega$. If $u_0\in Z$ satisfies
\begin{equation}\label{problema m}
\left\langle u_0,\varphi\right\rangle_Z-\int_\Omega m(x)u_0(x)\varphi(x)dx=0\quad\mbox{for any}\,\,\varphi\in Z,
\end{equation}
then $u_0= 0$.
\end{lem}
\begin{proof}
We can write $u_0=u_1+u_2$, where $u_1 \in H_k$ and $u_2\in \mathbb{P}_{k+1}$. By \eqref{problema m} we obtain
\begin{equation*}
\begin{alignedat}2
&\left\|u_1\right\|^2_Z=\int_\Omega m(x)(\left|u_1(x)\right|^2+u_2(x)u_1(x))dx\geq\int_\Omega (\lambda_k \left|u_1(x)\right|^2+m(x)u_2(x)u_1(x))dx,\\
&\left\|u_2\right\|^2_Z=\int_\Omega m(x)(u_1(x)u_2(x)+\left|u_2(x)\right|^2)dx\leq\int_\Omega (m(x)u_2(x)u_1(x)+\lambda_{k+1}\left|u_2(x)\right|^2)dx.
\end{alignedat}
\end{equation*}
If $u_0\neq 0$, then at least one of the above inequalities is strict and so, by using also \eqref{Hk} and \eqref{Pk}, it follows that
	\[\left\|u_1\right\|^{2}_{Z}-\left\|u_2\right\|^{2}_{Z}>\int_\Omega (\lambda_k \left|u_1(x)\right|^2-\lambda_{k+1}\left|u_2(x)\right|^2)dx\geq\left\|u_1\right\|^{2}_{Z}-\left\|u_2\right\|^{2}_{Z}
\]
which is a contradiction.
\end{proof}

\begin{prop}\label{prop palais}
Let $K$ and $f$ be two functions satisfying assumptions \eqref{K1}--\eqref{f1}. Moreover, assume there exists $k\in\mathbb{N}^*$ such that $\lambda_{k}<\underline{\alpha}(x)\leq\overline{\alpha}(x)<\lambda_{k+1}$ a.e. in $\Omega$. Let $\left\{u_{j}\right\}_{j\in\mathbb{N}}$ be a sequence in $Z$ such that $\left\{\mathcal J'(u_{j})\right\}_{j\in\mathbb{N}}$ is bounded.
Then, $\left\{u_{j}\right\}_{j\in\mathbb{N}}$ is bounded in $Z$.
\end{prop}
\begin{proof}

Step I) We argue by contradiction and suppose that $\left\{u_{j}\right\}_{j\in\mathbb{N}}$ is unbounded. By Lemma \ref{z}, up to a subsequence, there exists $u_0\in Z$ such that $u_j/\left\|u_{j}\right\|_{Z}$ converges to $u_0$ strongly in $L^2(\Omega)$ and a.e. in $\Omega$, as well as weakly in $Z$.

By our assumption on $\left\{\mathcal J'(u_{j})\right\}_{j\in\mathbb{N}}$ there exists a costant $c>0$ such that
\begin{equation}\label{limitatezza}
\begin{alignedat}2
&\frac{\left|\mathcal J'(u_j)(\varphi)\right|}{\left\|u_j\right\|_Z}  = \left|\left\langle \frac{u_j}{\left\|u_j\right\|_Z},\varphi\right\rangle_Z-\int_\Omega \frac{f(x, u_j(x))}{\left\|u_j\right\|_Z}\varphi(x)\,dx\right|\leq \frac{c\left\|\varphi\right\|_Z}{\left\|u_j\right\|_Z}
\end{alignedat}
\end{equation}
for any $\varphi\in Z$ and $j\in\mathbb{N}$.

By \eqref{f1} we get
\begin{equation*}
\frac{\left|f(x,u_{j})\right|}{\left\|u_{j}\right\|_{Z}}\leq \frac{a(x)}{\left\|u_{j}\right\|_{Z}}+b\frac{\left|u_{j}\right|}{\left\|u_{j}\right\|_{Z}},
\end{equation*}
where the sequence on the right-hand side is bounded in $L^2(\Omega)$.
So, there exists $\beta\in L^{2}(\Omega)$ such that, up to a subsequence,
$f(x,u_{j})/\left\|u_{j}\right\|_{Z}$ converges weakly to $\beta$ in $L^{2}(\Omega)$.

Now, we claim that
\begin{equation}\label{claim}
\begin{alignedat}2
&\beta(x)=m(x)u_0(x)\,\,\mbox{with}\,\, m\,\,\mbox{measurable and such that}\,\, \underline{\alpha}(x)\leq m(x)\leq\overline{\alpha}(x)\\
&\mbox{a.e. in}\,\, \Omega.
\end{alignedat}
\end{equation}
As is well known
\begin{equation*}
\liminf_{j\rightarrow+\infty}\frac{f(x,u_j(x))}{\left\|u_j\right\|_{Z}}\leq \beta(x)\leq\limsup_{j\rightarrow+\infty}\frac{f(x,u_j(x))}{\left\|u_j\right\|_{Z}}\quad\mbox{a.e. in}\,\,\Omega.
\end{equation*}
Moreover, if $x\in\Omega$ such that $u_0(x)\neq 0$, then for $j$ sufficiently large
\begin{equation*}
\frac{f(x,u_j(x))}{\left\|u_j\right\|_Z}=\frac{f(x,u_j(x))}{u_j(x)}\frac{u_j(x)}{\left\|u_j\right\|_Z}.
\end{equation*}
So, if $u_0(x)\geq 0$ we get $\underline{\alpha}(x)u_0(x)\leq \beta(x)\leq\overline{\alpha}(x)u_0(x)$, while if $u_0(x)<0$ the reversed inequalities hold true. This establishes \eqref{claim}, because when $x\in\Omega$ such that $u_0(x)=0$, we can set $m(x)=\displaystyle\frac{1}{2}(\underline{\alpha}(x)+\overline{\alpha}(x))$.

Thus, by sending $j\rightarrow+\infty$ in \eqref{limitatezza} and by using \eqref{claim} we get
\begin{equation*}
\left\langle u_0,\varphi\right\rangle_Z-\int_\Omega m(x)u_0(x)\varphi(x)dx=0\quad\mbox{for any}\,\,\varphi\in Z.
\end{equation*}
Thanks to this last formula and the fact that $\lambda_k<m(x)<\lambda_{k+1}$ we can use Lemma \ref{lemmanuovo} by obtaining $u_0=0$.

Step II) On the other hand, by using \eqref{limitatezza} with $\varphi=\displaystyle\frac{u_j}{\left\|u_j\right\|_Z}$ we get
\begin{equation*}
\left|1-\int_\Omega\frac{f(x,u_j(x))}{\left\|u_j\right\|_Z}\frac{u_j(x)}{\left\|u_j\right\|_Z}dx\right|\leq\frac{c}{\left\|u_j\right\|_Z},
\end{equation*}
for any $j\in\mathbb{N}$.
But, since $f(x,u_j)/\left\|u_j\right\|_{Z}$ is bounded in $L^{2}(\Omega)$, $u_j/\left\|u_j\right\|_Z$ converges to 0 in $L^2(\Omega)$ and $c/\left\|u_j\right\|_Z$ goes to 0, we get a contradiction.
\end{proof}

\begin{proof}[\bf Proof of Theorem \ref{Thsaddle}, when $\lambda_{k}<\underline{\alpha}(x)\leq\overline{\alpha}(x)<\lambda_{k+1}$]
At first, we prove that $\mathcal J$ satisfies the geometric structure required by the Saddle Point Theorem.
By Proposition \ref{prop Hk ortogonale} it follows that for any $M>0$ there exists $R>0$ such that if $u\in\mathbb{P}_{k+1}$ and $\left\|u\right\|_{Z}\geq R$, then $\mathcal J(u)\geq M$. If $u\in\mathbb{P}_{k+1}$ with $\left\|u\right\|_{Z}\leq R$,
by applying \eqref{f1}, \eqref{Pk}, and H\"{o}lder inequality we have
\begin{equation*}
\begin{alignedat}2
&\mathcal J(u)\geq -\int_{\Omega}F(x,u(x))dx\geq -\int_{\Omega}a(x)\left|u(x)\right|dx-\frac{b}{2}\int_{\Omega}\left|u(x)\right|^2dx\\
&\geq -\left\|a\right\|_{L^{2}(\Omega)}\left\|u\right\|_{L^{2}(\Omega)}-\frac{b}{2}\lambda^{-1}_{k+1}\left\|u\right\|^{2}_{Z}\geq -C_{R}
\end{alignedat}
\end{equation*}
for some constant $C_{R}=C(R, \Omega)>0$. So, we get
\begin{equation}
\mathcal J(u)\geq -C_{R}\quad\mbox{for any}\,\,u\in\mathbb{P}_{k+1}.
\end{equation}
By Proposition \ref{prop Hk} we can choose $T>0$ in such way that for any $u\in H_k$ with $\left\|u\right\|_{Z}=T$ we have
\begin{equation}
\sup_{u\in H_k,\,\left\|u\right\|_{Z}=T}\mathcal J(u)<-C_R\leq\inf_{u\in\mathbb{P}_{k+1}}\mathcal J(u),
\end{equation}
We have thus proved that $\mathcal J$ has the geometric structure of the Saddle Point Theorem (see \cite[Theorem 4.6]{rabinowitz}).
Now, it remains to check the validity of the Palais-Smale condition.
Let $c\in\mathbb{R}$ and let $\left\{u_{j}\right\}_{j\in\mathbb{N}}$ be a sequence in $Z$ such that
\begin{equation}\label{palais smale1}
\mathcal J(u_{j})\rightarrow c,
\end{equation}
and
\begin{equation}\label{palais smale2}
\sup\left\{\left|\mathcal J'(u_{j})(\varphi)\right|:\,\varphi\in Z,\,\left\|\varphi\right\|_{Z}=1\right\}\rightarrow 0\quad\mbox{for any}\,\,\varphi\in Z,
\end{equation}
as $j\rightarrow +\infty$.
By Proposition \ref{prop palais} $\left\{u_{j}\right\}_{j\in\mathbb{N}}$ is bounded, so by Lemma \ref{z}, up to a subsequence, there exists $u\in Z$ such that $u_{j}$ converges to $u$ strongly in $L^2(\Omega)$ and a.e. in $\Omega$, as well as weakly in $Z$.
Since, for any $\varphi\in Z$
\begin{equation*}
\mathcal J'(u_{j})(\varphi)=\left\langle u_j,\varphi\right\rangle_Z-\int_\Omega f(x,u_j(x))\varphi(x)dx,
\end{equation*}
by using also \eqref{f1} and \eqref{palais smale2} it follows that
\begin{equation}\label{1}
0=\left\|u\right\|^2_Z-\int_\Omega f(x,u(x))u(x)dx
\end{equation}
by taking $\varphi=u$, and also
\begin{equation}\label{2}
\left\|u_j\right\|^2_Z=\mathcal J'(u_{j})(u_j)+\int_\Omega f(x,u_j(x))u_j(x)dx\rightarrow\int_\Omega f(x,u(x))u(x)dx
\end{equation}
by taking $\varphi=u_j$ and sending $j\rightarrow+\infty$.
Indeed, for the last formula we observe that
\begin{equation*}
\left|f(x,u_j)u_j\right|\leq a(x)\left|u_j\right|+b\left|u_j\right|^2,
\end{equation*}
where the sequence on the right-hand side converges in $L^1(\Omega)$.

Thus, by combining \eqref{1} and \eqref{2} we get $\left\|u_j\right\|_Z\rightarrow\left\|u\right\|_Z$ and so $\left\{u_j\right\}_{j\in\mathbb{N}}$ converges strongly to $u$ in $Z$.
\end{proof}

\begin{proof}[\bf Proof of Corollary \ref{cor}]
Let $u_1$, $u_2\in Z$ be two solutions of problem \eqref{wf}. Then, we set $w:=u_1-u_2$ and
\begin{equation*}
m(x):=
\left\{\begin{array}{ll}
$$\displaystyle\frac{f(x,u_1(x))-f(x,u_2(x))}{u_1(x)-u_2(x)} & \mbox{if } u_1(x)\neq u_2(x),$$\\

$$\displaystyle\frac{1}{2}(\lambda_k +\lambda_{k+1}) & \mbox{if } u_1(x)=u_2(x).$$
\end{array}\right.
\end{equation*}
So, $m$ is a measurable function which verifies $\lambda_{k}< m(x)<\lambda_{k+1}$ a.e. in $\Omega$ thanks to \eqref{f2}. Moreover, \eqref{wf} implies that
\begin{equation*}
\left\langle w,\varphi\right\rangle_Z-\int_\Omega m(x)w(x)\varphi(x)dx=0\quad\mbox{for any}\,\,\varphi\in Z.
\end{equation*}
Thus, by Lemma \ref{lemmanuovo} it follows that $w= 0$.
\end{proof}


\begin{thebibliography}{99}

\bibitem{BFV} B. Barrios, A. Figalli, and E. Valdinoci, {\em Bootstrap regularity for integro-differential operators and its application to nonlocal minimal surfaces}, to appear in Ann. Scuola Norm. Sup. Pisa Cl. Sci. \textbf{5}, available online at {\tt http://arxiv.org/abs/1202.4606}\,.

\bibitem{brezis} {\sc H. Br\'ezis}, Analyse fonctionelle. Th\'{e}orie et
applications, {\em Masson}, Paris (1983).

\bibitem{valpal}{\sc E. Di Nezza, G. Palatucci and E. Valdinoci}, {\em Hitchhiker's guide to the fractional Sobolev spaces},
B. Sci. Math. \textbf{136} (2012), no. 5, 521--573.

\bibitem{Mugnai} {\sc D. Mugnai}, Lecture Notes,
Universit\`a di Perugia, available on line at {\tt http://www.dmi.unipg.it/~mugnai/download/corso0.pdf}\,.

\bibitem{rabinowitz} {\sc P.H. Rabinowitz}, Minimax methods in critical point theory with applications to differential equations,  {\em CBMS Reg. Conf. Ser. Math.}, 65, {\em American Mathematical Society}, Providence, RI (1986).

\bibitem{sv4}{\sc R. Servadei}, {\em The Yamabe problem in a non--local setting},
to appear in Adv. Nonlinear Anal. \textbf{3}, available at {\tt http://www.ma.utexas.edu/mp$\_$arc-bin/mpa?yn=12-40}\,.

\bibitem{svnew}{\sc R. Servadei and E. Valdinoci}, {\em Lewy-Stampacchia type estimates for variational inequalities driven by (non)local operators}, Rev. Mat. Iberoam. \textbf{29} (2013), no.~3, 1091--1126.

\bibitem{sv2}{\sc R. Servadei and E. Valdinoci}, {\em Mountain Pass solutions for non--local elliptic operators}, J. Math. Anal. Appl. \textbf{389} (2012), no.~2, 887--898.

\bibitem{sv3}{\sc R. Servadei and E. Valdinoci}, {\em The Brezis--Nirenberg result for the fractional Laplacian}, to appear in Trans. Amer. Math. Soc., available online at {\tt http://www.math.utexas.edu/mp$\_$arc-bin/mpa?yn=11-196}\,.

\bibitem{sv5}{\sc R. Servadei and E. Valdinoci}, {\em Variational methods
for non-local operators of elliptic type},
Discrete Contin. Dyn. Syst. \textbf{33} (2013), no.~5, 2105--2137.

\bibitem{sv6}{\sc R. Servadei and E. Valdinoci}, {\em Weak and viscosity solutions of the fractional Laplace equation}, to appear in Publ. Mat., available online at {\tt www.ma.utexas.edu/mp$\_$arc/c/12/12-82.ps.gz}\,.

\end{thebibliography}
\end{document}